\documentclass[a4paper,oneside,onecolumn,preprint,number,10pt]{elsarticle}
\usepackage[utf8]{inputenc}
\usepackage{array,booktabs,color,amsmath,amssymb,amsthm,comment,pifont}
\renewcommand{\labelenumi}{{\sl (\roman{enumi})}}
\newtheorem{theorem}{Theorem}

\newtheorem{proposition}[theorem]{Proposition}
\theoremstyle{remark}
\newtheorem{ex}[theorem]{Example}

\theoremstyle{definition}
\newtheorem{definition}[theorem]{Definition}

\newcommand{\cb}    [1]{\ensuremath{\left  \{      #1  \right \}       }}

\newcommand{\of}    [1]{\ensuremath{\left (        #1  \right )        }}

\newcommand{\st} {\ensuremath{|\;}}

\newcommand{\Int} {{\rm int \,}}
\newcommand{\conv}  {{\rm conv \,}}

\newcommand{\cone}{{\rm cone\,}}

\DeclareMathOperator*{\Min}{ Min}

\newcommand{\I}{\mathcal{I}}

\newcommand{\D}{\mathcal{D}}

\renewcommand{\P}{\mathcal{P}}

\newcommand{\R}{\mathbb{R}}

\newcommand{\smz}{\!\setminus\!\{0\}}

\newcommand*{\bensolve}{{\em Bensolve}}
\newcommand*{\oldsolve}{{\em Bensolve~v1.2}}

\newcommand*{\leqc}{\mathbin{\leqslant_C}}

\newcommand*{\sclobj}{c}
\newcommand*{\sclrstr}{\rstr}

\newcommand*{\rstr}{S}

\DeclareMathOperator{\rank}{rank}

\journal{ejor}

\begin{document}
\begin{frontmatter}
  
  \title{The vector linear program solver {\sl Bensolve} -- notes on
    theoretical background}
  \author{Andreas Löhne}
  \ead{andreas.loehne@uni-jena.de}
  \address{Friedrich Schiller University Jena\\ Department of
    Mathematics\\ 07737 Jena\\ Germany}
  \author{Benjamin Weißing\corref{cor1}}
  \ead{benjamin.weissing@mathematik.uni-halle.de}
  \address{Martin Luther University Halle--Wittenberg\\ Department of
    Mathematics\\ 06099 Halle (Saale)\\ Germany}
  \cortext[cor1]{corresponding author}

  \begin{keyword}
    vector linear programming \sep linear vector optimization \sep
    multiple objective optimization
    \MSC[2010] 90C29 \sep 90C05 \sep 52B55 \sep 15A39
  \end{keyword}
  \begin{abstract} 
    \bensolve\ is an open source implementation of Benson's algorithm and
    its dual variant. Both algorithms compute primal and dual solutions of
    vector linear programs (VLP), which include the subclass of multiple
    objective linear programs (MOLP). The recent version of \bensolve\ can
    treat arbitrary vector linear programs whose upper image does not
    contain lines. This article surveys the theoretical background of the
    implementation. In particular, the role of VLP duality for the
    implementation is pointed out. Some numerical examples are
    provided. In contrast to the existing literature we consider a
    less restrictive class of vector linear programs.
  \end{abstract}
\end{frontmatter}

\section{Introduction}
Let us start with the formulation of a {\em multiple objective linear
  program}, which is a special case of a {\em vector linear program}. Given
positive integers $n, m, q$ and data of the form
\begin{align*}
  &P \in \R^{q \times n}\text{,}\\
  &B \in \R^{m \times n},
  a \in \bigl(\R\cup\{-\infty\}\bigr)^{m\times 1}, b \in \bigl(\R\cup\{+\infty\}\bigr)^{m
    \times 1}\text{,}\\
  \intertext{and}
  &l \in \bigl(\R\cup\{-\infty\}\bigr)^{n\times 1},
  s \in \bigl(\R\cup\{+\infty\}\bigr)^{n\times 1} \text{,}
\end{align*}
we consider the problem
\begin{equation}\label{molp}
  \tag{MOLP} \min Px \quad \text{ subject to }\quad a \leq B x \leq
  b,\quad l \leq x \leq s.
\end{equation}
Minimization is understood with respect to the component-wise ordering
in $\R^q$.

In the more general setting of a vector linear program, one
considers a partial ordering $\leq_C$ generated by a polyhedral convex
ordering cone $C\subseteq \R^q$ that does not contain lines, or
equivalently in this setting, is pointed. For vectors $y,z \in \R^q$,
we define
\begin{equation*}\label{ordering}
  y \leq_C z \quad\mathrel{\mathop:}\iff\quad z-y \in C.
\end{equation*}
The polyhedral cone $C$ can be given either by a matrix $Y \in
\R^{q\times o}$ as
\begin{equation}\label{rep_cone}
  C=\cb{y \in \R^q\st \exists v \in \R^o: y = Y v,\; v \geq 0},
\end{equation}
or by a matrix $Z\in \R^{q\times p}$ as
\begin{equation}\label{rep_dualcone}
  C=\cb{y \in \R^q\st Z^T y \geq 0}.
\end{equation}
Throughout, \eqref{rep_cone} is referred to as
\texttt{cone}-representation and \eqref{rep_dualcone} is called
\texttt{dualcone}-representation. This can be motivated as $C$ is
generated by the columns of the matrix $Y$ (in the sense of
\eqref{rep_cone}), whereas the dual cone
$$C^*:=\cb{w \in \R^q \st \forall y \in C: y^T w \geq 0}$$ of $C$ is
generated by the columns of the matrix $Z$. The latter statement is a
consequence of Farkas' lemma. The resulting vector linear program is
\begin{equation}\label{vlp}
  \tag{VLP} \textstyle{\min_C} Px \quad \text{ subject to }\quad a
  \leq B x \leq b,\quad l \leq x \leq s.
\end{equation}

The recent version of {\sl Bensolve} assumes further that the cone $C$
is {\em solid}, i.e., it has dimension $q$, or equivalently, nonempty
interior. This assumption is equivalent to $\rank Y = q$. Further, $C$
is free of lines if and only if $\rank Z = q$. This implies that
necessarily $o \geq q$ and $p \geq q$.

We close this section with some bibliographical remarks. It seems that
Dauer \cite{Dauer87} and Dauer \& Liu \cite{DauLiu90} started to
propagate the objective space approach in multiple objective linear
programming, saying that on the one hand it is more efficient, and on
the other hand it is sufficient in practice to generate only the
minimal (or non-dominated) vertices in the objective space (or image
space). The algorithms of {\sl Bensolve} are based on the papers by
Benson \cite{Benson98a, Benson98} from 1998. Since then, several
extensions, simplifications and improvements have been published, among
them the introduction of a dual algorithm in \cite{EhrLoeSha12}, the
extension to the unbounded case in \cite{Loehne11}, approximate
variants in \cite{ShaEhr08}, the extension to pointed solid polyhedral
ordering cones in \cite{HamLoeRud13}, a simplification which requires
only one LP per iteration (independently) in \cite{Csirmaz15} and
\cite{HamLoeRud13}, and a complexity analysis in
\cite{BoeMut15}. Recently, it was shown that {\sl Bensolve} can be
used to compute projections of polyhedra \cite{LoeWei15}. In contrast to the literature, see e.g.\ \cite{HamLoeRud13}, we make weaker assumptions: The geometric duality parameter vector $c\in \R^q$ is required to satisfy $c_q\neq 0$ rather than $c_q = 1$. The choice of ordering cones is therefore less restricted. Moreover, we derive geometric dual programs for maximization problems as well as for problems with box constraints.

\section{Solution concepts}
In contrast to most classical textbooks on vector optimization and
multiple objective programming, \bensolve\ is based on a
recent solution concept, which has been introduced in \cite{HeyLoe11}
and \cite{Loehne11}.  This concept entails the notions of minimality
\emph{and} infimum attainment, which are equivalent in the scalar
linear programming case, but do diverge with increasing image space
dimension.  While in the classical literature essentially only
minimality is taken into account, we will show in the following that
\emph{both} conceptions are essential for a proper solution concept.
To this end, we start with the classical scalar theory and extend it
to the case of multiple image space dimensions.\par

Consider the scalar linear program
\begin{equation}\tag{LP}
  \label{eq.linprog}
  \min x\mapsto \sclobj x\quad\text{s.t. $x \in \sclrstr$,}
\end{equation}
where $\sclrstr \subseteq \R^n$ is some feasible set defined by linear
inequalities.  Finding a \emph{solution} to \eqref{eq.linprog} means
to find a feasible variable $\bar{x}\in\sclrstr$ whose image $\sclobj
\bar{x}$ equals the minimum of the set of objective values
$\left\{\sclobj x\mid x\in \sclrstr
\right\}=\mathrel{\mathop:}c\left[X\right]$.\par
The generalization of minimality to higher image space dimension with
respect to the ordering induced by $C$ is a standard notion:
\begin{definition}
  A point $y \in A\subseteq \R^q$ is called \emph{$C$-minimal} in $A$,
  if for every point $\tilde{y} \in A$ the following implication
  holds:
  \begin{equation*}
    \tilde{y}\leqc y \Rightarrow \tilde{y}=y\text{.}
  \end{equation*}
  The set of all $C$-minimal points of a set $A$ is denoted by $\Min_C
  A$.\par
\end{definition}
The generalization of the term 'solution', however, is not
so obvious. Likewise to the scalar case, a solution to a problem
should be a single entity. But, a single minimizer is \emph{not} an
appropriate solution, even though it is referred to as \emph{efficient
  solution} in the classical literature. The reason is that a single
minimizer can be obtained by solving a scalarized problem, which in
the present setting is usually a linear program. The problem to
determine a single minimizer is therefore a matter of scalar
optimization. Thus, the \emph{single entity} mentioned above should be
a \emph{set} of minimizers.\par

The problem to compute \emph{all} extreme minimizers has been investigated,
see \cite{EvaSte73, Steuer89}, but for many applications, this
approach does not seem to be tractable.  This can be seen already by
considering the scalar case:  The problem to compute all extreme
minimizers (i.e.\ a representation of all solutions) of a linear
program is NP hard: Consider the feasibility problem. This leads to
vertex enumeration, which is NP hard, see e.g.\ \cite[Proposition
5]{Dyer}.\par
Dauer \cite{Dauer87} and many followers aimed at characterizing
\emph{every} minimal element in the \emph{objective space}. A
corresponding set of minimizers is what we understand to be a
\emph{solution} to a vector linear program.  Surprisingly this idea can
be defined via \emph{infimum attainment}.\par

Let us consider the scalar case first.  An alternative
characterization of a solution for \eqref{eq.linprog}, which does not
rely on minimality, is provided by means of infimum attainment: The
image $c\bar{x}$ of a solution $\bar{x}$ to \eqref{eq.linprog} is the
infimum of the image of the feasible set $c[S]$, which may be expressed
equivalently by
\begin{equation}\label{eq.scal_inf}
    \left\{\sclobj \bar{x}\right\}+\R_{{}+{}} \subseteq \sclobj\left[\sclrstr\right]+\R_{{}+{}}\text{.}
\end{equation}
The generalization of the right hand side of inequality
\eqref{eq.scal_inf} in the case of multiple image space dimensions
features prominently in the solution concept for \eqref{vlp} and is
therefore denoted by a dedicated term:
\begin{definition}
  The \emph{upper image} of \eqref{vlp} is the Minkowski sum of the
  image $P[S]$ of the feasible set $\rstr$ and the ordering cone $C$:
  \begin{equation}
    \P \mathrel{\mathop:}= P\left[S \right] + C\text{.}
  \end{equation}
  The \emph{lower image}, in contrast, is obtained by adding the
  negative of the ordering cone to the image:
  \begin{equation}
    P\left[S \right] - C\text{.}
  \end{equation}
  As a generic term for both notions we use \emph{extended image}.
\end{definition}
The upper image can be understood as an infimum, too: It serves as
infimum in a complete lattice which embeds the image space $\R^q$.  As the
theoretical details are beyond the scope of the present article, the
interested reader is referred to \cite{HeyLoe11, Loehne11} for a
thorough discussion.  Here, we focus on the practical implications the
consideration of the upper image yields.  Observe that the right hand
side of inequality \eqref{eq.scal_inf} describes all the information
in the image space that matters for a solution to \eqref{eq.linprog}:
While the image $\sclobj [\sclrstr ]$ may be a closed interval
$\left[\sclobj \bar{x} , y\right]$, the actual value of $y$ or even
the existence of $y$ is not of interest for a minimization problem.
Adding the ordering cone ($\R_{{}+{}}$ in the case of scalar
minimization) to the image $c[S]$ of the feasible set conveys this
fact. Before showing that this feature of the upper image is retained
in higher image space dimensions (Proposition \ref{prop_1}), we
provide the notion of minimal directions in order to be able to
describe the upper image geometrically. Hitherto we recapitulate that
polyhedra may be represented by different means:
\begin{enumerate}
\item
  In terms of finitely many points and directions, called V-representation.
\item
  As intersection of finitely many affine halfspaces, called
  H-representation.
\end{enumerate}
\begin{definition}
  $y^h \in \R^q\smz$ is a \emph{$C$-minimal direction} in
  $A\subseteq\R^q$ if for some $y\in A$ the point $y+\mu y^h$ is
  $C$-minimal in $A$ for every scalar $\mu\geq 0$.
\end{definition}
Minimizers now can be defined as feasible elements generating minimal
points and minimal directions. We set $\ker P := \left\{ x \in \R^n \st Px = 0 \right\}$.
\begin{definition} [Feasibility]\leavevmode\par
  A \emph{point} $\bar{x} \in \R^n$ is called \emph{ feasible} for
  \eqref{vlp} if
  \begin{equation*}
    \bar{x}\in S \mathrel{\mathop:}= \left\{ x\in \R^n \mid a\leq Bx
    \leq b\,\text{,}\; l\leq x \leq s\right\}\text{.}
  \end{equation*}\par
  A \emph{direction} $\bar{x}^h \in \R^n \setminus \ker P$ is called
  \emph{feasible} for \eqref{vlp} if
  \begin{equation*}
    \bar{x}^h \in S^h \mathrel{\mathop:}= \left\{ x\in \R^n \mid
    0\cdot a\leq Bx \leq 0\cdot b\,\text{,}\; 0\cdot l\leq x \leq
    0\cdot s\right\}\text{,}
  \end{equation*}
  where we define $0 \cdot \pm\infty = \pm\infty$.\par
\end{definition}
\begin{definition}[Minimizer]
  A feasible point $x \in S$ is called \emph{minimizer}
  for \eqref{vlp} if its image $Px$ is $C$-minimal in
  $P[S]$. Likewise, a direction $x^h \in S^h\setminus \ker P $ is
  called \emph{minimizer} for \eqref{vlp} if $Px^h$ is a
  $C$-minimal direction in $P[S]$.
\end{definition}
An infimizer of \eqref{vlp} is a pair of sets that generates the
upper image.  For a set $B\subseteq\R^q$, $\conv B$ denotes its convex
hull.  The cone generated by a non-empty set $B\subseteq\R^q$ is
denoted by $\cone B\mathrel{\mathop{:}}=\left\{tx\st x\in
B,\,t\geqslant 0 \right\}$.
We set $\cone\emptyset\mathrel{\mathop{:}}=\{0\}$.
\begin{definition}[Infimizer]\label{dfn.infmsr}
  A pair $(\bar S, \bar S^h)$, where $\bar S$ is a nonempty finite set
  of feasible points and $\bar S^h$ is a finite set of feasible
  directions is called a {\em finite infimizer} of \eqref{vlp} if
  \begin{equation}\label{inf_att}
    \conv P[\bar S] + \cone P[\bar S^h] + C \supseteq P[S] + C.
  \end{equation}	
\end{definition}
We note here that \emph{minimality} is a kind of ``local'' property,
whereas \emph{infimum attainment} characterizes the relevant parts of
the image of \eqref{vlp} ``as a whole.''  A solution to \eqref{vlp}
combines both conceptions:
\begin{definition}[Solution]
  A finite infimizer $(\bar S, \bar S^h)$ is said to be a {\em
    solution} of \eqref{vlp} if the sets $\bar S$ and $\bar S^h$
  consist of minimizers only.
\end{definition}
It is imperative, of course, that no minimal point is lost during the
transition from the image $P\left[S \right]$ to the upper image $\P$.
This is ensured by the following proposition:
\begin{proposition} \label{prop_1}
  Suppose $\P$ contains a vertex.  Then
  \begin{enumerate}
  \item $\Min_C \P = \Min_C P[S]$.
  \item Every vertex of $\P$ is $C$-minimal.
  \item An extreme direction of $\P$ is $C$-minimal if and only if it
    does not belong to $C$.
  \end{enumerate}	
\end{proposition}

\begin{proof}\leavevmode
  The first two statements are well-known and can be proven
  straight-forward.  To prove \labelenumi, first note that no $c \in
  C\smz$ may serve as a minimal direction, as for any $y\in \P$ we
  have
  \begin{equation*}
    y \leqc y+c \quad\text{and}\quad y\neq y+c\text{.}
  \end{equation*}
  Now let an extremal direction $\bar{y} \in \P$ be given, that is,
  there exists $y\in \P$ such that
  $F\mathrel{\mathop:}=\left\{y+t\bar{y}\mid t\geq 0\right\}$ is a $1$-dimensional face of $\P$.  Now consider an
  element $\tilde{y}\in \P$ with $\tilde{y}\leqc y+t\bar{y}$ for given
  $t\geq 0$, which means that there exists $c \in C$ with
  $\tilde{y}+c=y + t\bar{y}$.  For any given $\lambda \in (0,1)$ we
  have $\tilde{y}+\frac{1}{1-\lambda}c\in\P$.  From $y+t\bar{y}
  = \lambda\tilde{y}+\left(1-\lambda\right)\left(\tilde{y}+\frac{1}{1-\lambda}c\right)\in
  F$  we conclude $\tilde{y}\in F$ and $\tilde{y}+\frac{1}{1-\lambda}c \in
  F$.  This means either $\bar{y}$ is a nonnegative (as $\P$ does not
  contain lines) multiple of $c$ or $c=0$.  Therefore, if
  $\bar{y}\notin C$, $y + t\bar{y}$ is $C$-minimal for every $t\geq
  0$, proving minimality of $\bar{y}$.
\end{proof}

The preceding proposition shows that for obtaining a solution to
\eqref{vlp} it is sufficient to find the vertices $Px$ and minimal
extreme directions $Px^h$ of the upper image as well as preimages
$x\in S, x^h\in S^h$ which generate them.  In fact, this is how
\bensolve\ works: A vertex representation of the upper image is found
along with corresponding preimages for the vertices and minimal
directions.  Actually, \bensolve\ does even more: The V-representation
found by \bensolve\ is \emph{irreducible}, that is, if any element of
the solution is left out, it will not generate the upper image
$\mathcal{P}$ in the sense of \eqref{dfn.infmsr}.\par
Those extremal directions in
the V-representation which are not minimal (namely those belonging to
$C$) are not a part of the solution, but nevertheless they are
necessary for a complete description of the upper image.  To
emphasize their significance, these directions are labeled
specifically:
\begin{definition}
  The set of extreme directions of $\P$ belonging to $C$ is called
  \emph{cone compartment}.
\end{definition}

\section{Dual problem and dual solutions}

Duality plays an important role for {\sl Bensolve}.  For the user of
the software the following aspects are important:

\begin{enumerate}
	\item The extended image $\D$ of the dual problem contains
          information about the extended image $\P$ of the primal
          problem: {\sl Bensolve} outputs a V-representation of $\D$,
          which can be used to obtain an H-representation of
          $\P$. Likewise, an H-representation of $\D$ can be obtained
          from the V-representation of $\P$ computed by {\em
            Bensolve}.
	\item A dual algorithm can be chosen, which can be
          advantageous for certain problem instances. The dual
          algorithm constructs an outer and inner approximation of
          $\D$ which corresponds, by duality, to an inner and outer
          approximation of $\P$.
	\item A duality parameter vector $c\in\R^q$ can be chosen by the
          user. The dual problem and the dual solution (but not the
          primal problem and primal solution) depend on this
          vector. It has influence on numerical issues of both the
          primal and the dual algorithm.
\end{enumerate}

To introduce the reader into the main ideas of duality for vector
linear programming (in the sense of ``geometric duality'' established
in \cite{HeyLoe08}), we begin with the following special case of
\eqref{molp}:
\begin{equation}\label{molp1}
\min Px \quad \text{ subject to }\quad a \leq B x.
\end{equation}
The dual problem to \eqref{molp1} is the following vector linear
program with ordering cone $K:=\{y \in \R^q |\; y_1 = 0, \dots,
y_{q-1}= 0, y_q \geq 0\}$ :
\begin{equation}\label{molp1_d}
	\text{$K$-maximize} \quad D(u,w) \quad { s.t. } \quad B^T u =
        P^T w, \quad u \geq 0,\quad w \geq 0,\quad e^T w = 1,
\end{equation}
where the objective function is defined as
$$ D: \R^m \times\R^q \to \R^q,\quad D(u,w) := \of{w_1,\,
  w_2,\,\dots,\,w_{q-1},\,a^T u}^T$$ and $e=(1,\dots,1)^T$ denotes the
all-one vector in $\R^q$. In this special case, the duality parameter
vector $c \in \R^q$ has been chosen as $c=e$. Later on, the constraint
$e^T w = 1$ will be replaced by $c^T w = 1$. The feasible set is
denoted by
$$ T:= \cb{(u,w) \in \R^m \times \R^q \st B^T u = P^T w, \; u \geq
  0,\; w \geq 0,\; e^T w = 1}.$$
Duality provides a relationship between the upper image $\P$ of the
primal problem \eqref{molp1} and the lower image $\D$ of the dual
problem \eqref{molp1_d}, defined as
$$ \D := D(T) - K\text{.}$$
In addition to weak and strong duality, for
details see e.g.\ \cite{Loehne11}, there is a third type of duality
relation, called {\em geometric duality}. It states that there is a
one-to-one correspondence between the proper faces of the polyhedron
$\P$ and the non-vertical (i.e.\ the last component of the
corresponding normal vector does not vanish) proper faces of the
polyhedron $\D$. The dimension of a proper face of $\P$ and the
dimension of the corresponding face of $\D$ add up to $q-1$. In
particular, facets ($(q-1)$-dimensional faces) correspond to vertices
(0-dimensional faces).\par
In order to be able to formulate duality results we consider the following
bi-affine {\em coupling function}, which was introduced in \cite{HeyLoe08}:
\begin{equation}
	\varphi(y,y^*) : = \sum_{i=1}^{q-1} y_i y^*_i + y_q
        \of{1-\sum_{i=1}^{q-1} y^*_i} - \xi(y) y^*_q,
\end{equation}
where
$$ \xi(y) = \left\{\begin{array}{l} 1 \text{ if $y$ is a point} \\ 0
\text{ if $y$ is a direction.}
\end{array}\right.$$\par

The next theorem is a consequence of the geometric duality theorem
\cite[Theorem 3.1]{HeyLoe08}. It points out the facts relevant for
users of {\sl Bensolve} and does not aim to cover the complete idea of
geometric duality. For simplicity, we assume that $\P$ has a vertex,
which corresponds exactly to the setting of the recent version of {\sl
  Bensolve} \cite{bensolve}.

\begin{theorem} \label{th1} If $\P$ has a vertex, then the following statements hold true:
	\begin{enumerate}
		\item A finite set $\bar Y$ of points and directions
                  in $\R^q$ is an irredundant V-representation of $\P$
                  if and only if
	$$ \varphi(y,y^*) \geq 0,\quad y \in \bar Y$$ is an
                  irredundant H-representation of $\D$.
	\item A finite set $\bar W$ of points combined with the
          direction $\of{0,\dots,0,-1}^T$ in $\R^q$ forms an irredundant
          V-representation of $\D$ if and only if
	$$ \varphi(y,y^*) \geq 0,\quad y^* \in \bar W$$ is an
          irredundant H-representation of $\P$.
	\end{enumerate}
\end{theorem}
\begin{proof}
	Statement {\sl ({\romannumeral 1})}\ follows from Corollary 3.3 in \cite{HeyLoe08}
        and Theorem 4.62 in \cite{Loehne11}. One has to take into
        account the following facts: (a) If a polyhedron $P$ has a
        vertex, then an irredundant V-representation of $P$ consists
        exactly of all vertices and all extreme directions of $P$, see
        e.g. \cite{Schrijver}; (b) Every vertex of $\P$ is weakly
        minimal (and even minimal), see e.g. \cite[Corollary
          4.67]{Loehne11}; (c) The vertical facets are exactly the
        ones which are not $K$-maximal, see e.g. \cite[Lemma
          4.60]{Loehne11}.\par
        	
	Statement {\sl ({\romannumeral 2})}\ follows from Corollary 3.2 in \cite{HeyLoe08}
        and the following facts: (d) $\D$ has a vertex, see
        \cite[Lemma 5.2]{HeyLoe08}, hence an irredundant
        V-representation of $\D$ consists of its vertices and extreme
        directions; (e) every proper face and hence any facet of $\P$
        is weakly minimal, see e.g. \cite[Lemma 5.6]{HeyLoe08}.
\end{proof}
Let us now consider the general case of \eqref{vlp}. Given a (fixed)
duality parameter $c \in \R^q$ with $c_q\neq 0$, the dual problem of \eqref{vlp} is the
following vector linear program with ordering cone $K\mathrel{\mathop:}=\{y \in \R^q
\mid y_1 = 0, \dots, y_{q-1}= 0, y_q \geq 0\}$:
\begin{equation}\label{molp_d}
	\text{$K$-maximize}\quad D(u,w,v) \quad \text{ s.t. } \quad
        B^T u = P^T w + v,\quad Y w \geq 0,\quad c^T w = 1
\end{equation}
with objective function
\begin{equation}\label{obj_d2}
	D(u,w,v) = \of{ \frac{c_q}{|c_q|} w_1,\; \frac{c_q}{|c_q|}
          w_2,\; \dots,\; \frac{c_q}{|c_q|} w_{q-1},\; d(u,v)}^T.
\end{equation}
The function $d:\R^m \times\R^n \to \R$ is defined as
\begin{equation}
		d(u, v) = a^T u^+ - b^T u^- + l^T v^- - s^T v^+,
\end{equation}
where we define $\pm \infty \cdot 0 = 0$ and
\begin{gather*}
        \alpha^{+}
        \mathrel{\mathop:}= \max(0,\alpha),\qquad \alpha^{-}:=\max(-\alpha,0)\text{,}\\
        \intertext{for $\alpha\in\R$, and}
(\alpha_1,\dots,\alpha_m)^{+} \mathrel{\mathop:}=
(\alpha_1^+,\dots,\alpha_m^+),\qquad (\alpha_1,\dots,\alpha_m)^{-} \mathrel{\mathop:}=
(\alpha_1^-,\dots,\alpha_m^-)
\end{gather*}
for $(\alpha_1,\dots,\alpha_m)\in\R^m$.
\begin{theorem} \label{th2} Let $c \in \Int C$ such that $c_q \neq 0$. Then Theorem \ref{th1} also holds for the general case of problem \eqref{vlp} and the dual problem \eqref{molp_d} if the following generalized coupling function is used:
\begin{equation}\label{phi1}
	\varphi(y,y^*) : = c_q \sum_{i=1}^{q-1} y_i y^*_i + y_q
        \of{\frac{|c_q|}{c_q} -\sum_{i=1}^{q-1} c_i y^*_i} - \xi(y)
        |c_q| y^*_q.
\end{equation}
\end{theorem}
\begin{proof} The generalization of geometric duality to the case of a polyhedral convex ordering cone $C$ that does not contain lines and has nonempty interior and a duality parameter vector $c \in \Int C$ with $c _q = 1$ can be found in \cite{HamLoeRud13}. In this setting the dual problem is
\begin{equation*}
	\text{$K$-maximize} \quad D(u,w) \quad { s.t. } \quad B^T u =
        P^T w, \quad u \geq 0,\quad Y^T w \geq 0,\quad c^T w = 1
\end{equation*}
with objective function $D(u,w) \mathrel{\mathop:}= \of{w_1,\,
  w_2,\,\dots,\,w_{q-1},\,a^T u}^T$. The coupling function in this
setting is
\begin{equation}
	\varphi(y,y^*) \mathrel{\mathop:} = \sum_{i=1}^{q-1} y_i y^*_i + y_q \of{1 -
          \sum_{i=1}^{q-1} c_i y^*_i} - \xi(y) y^*_q.
\end{equation}\par

Now let us relax the assumption $c_q = 1$ by
$c_q>0$. Introducing new coordinates in the image space of the primal
problem by replacing the last component $y_q \to c_q^{-1} y_q$ (which
effects the data $Y$ and $P$ of the primal problem as well as the
duality parameter vector $c$), we obtain the dual problem
\begin{equation}
	\text{$K$-maximize} \; D(u,\bar w) \quad { s.t. } \; B^T
        u = P^T \bar w, \; u \geq 0,\; Y^T \bar w \geq 0,\;
        c^T \bar w = 1,
\end{equation}
with objective function $D(u,\bar w) \mathrel{\mathop:}= \of{\bar w_1,\,\dots,\,\bar
  w_{q-1},\,a^T u}^T$, where the dual variable $w \in \R^q$ has been
replaced by $\bar w$ with $\bar w_i \mathrel{\mathop:}= w_i$ for $i=1\dots,q-1$ and
$\bar w_q \mathrel{\mathop:}= c_q^{-1} w_q$. The coupling function is
\begin{equation}
	\varphi(y,y^*) : = \sum_{i=1}^{q-1} y_i y^*_i + c^{-1}_q y_q
        \of{1 - \sum_{i=1}^{q-1} c_i y^*_i} - \xi(y) y^*_q.
\end{equation}
Of course, the statement of Theorem \ref{th1} remains valid if the
coupling function is replaced by \eqref{phi1} and the objective
function is defined by \eqref{obj_d2} for the special case $d(u,v) =
a^T u$.
	
Let us now consider the case $c_q < 0$. We perform a coordinate
transformation $ y \to -y$ which results in a primal problem with data
$\bar P \mathrel{\mathop:}= -P$, $\bar Y \mathrel{\mathop:}= -Y$ and a duality parameter $\bar c \mathrel{\mathop:}=
-c$. Since $\bar c_q > 0$, the result is known for this case from the
first part of the proof. The dual problem is
\begin{equation*}
	\text{$K$-maximize} \quad D(u,\bar w) \quad { s.t. } \quad B^T
        u = \bar P^T \bar w, \quad u \geq 0,\quad \bar Y^T \bar w \geq
        0,\quad \bar c^T \bar w = 1
\end{equation*}
with objective function $D(u,\bar w) \mathrel{\mathop:}= \of{\bar w_1,\,\dots,\,\bar
  w_{q-1},\,a^T u}^T$. Substitution of the dual variable $\bar w$ by
$-w$ leads to
\begin{equation*}
	\text{$K$-maximize} \quad D(u, w) \quad { s.t. } \quad B^T u =
        P^T w, \quad u \geq 0,\quad Y^T w \geq 0,\quad c^T w = 1,
\end{equation*}
with objective function $D(u, w) \mathrel{\mathop:}= \of{- w_1,\, - w_2,\,\dots,\,-
  w_{q-1},\,a^T u}^T$. The coordinate transformation ($y \to -y$, $c
\to -c$) transforms the coupling function \eqref{phi1} for the case
$c_q > 0$ to \eqref{phi1} for the case $c_q < 0$.

The constraints of the general case \eqref{vlp} can be expressed as
\begin{equation}\label{constraints_3}
	\begin{pmatrix}
		B\\-B\\I\\-I
	\end{pmatrix} x \geq \begin{pmatrix}
		a\\-b\\l\\-s
	\end{pmatrix},
\end{equation}
where $I$ denotes the $n\times n$ unit matrix. This leads to the dual
constraints
$$B^T u' - B^T u'' = P^T w + v' - v'',\quad u',u'',v',v'' \geq 0,\quad
Y^T w \geq 0,\quad c^T w = 1,$$ and the last component of the
objective function $D$ is
$$ d(u,v) = a^T u' - b^T u'' + l ^T v'' - s^T v'.$$ If some components
of the right-hand side in \eqref{constraints_3} are $-\infty$, the
corresponding dual variables must vanish (i.e. they do not occur in
the dual program). This is taken into account by setting $\pm \infty
\cdot 0 = 0$.

Finally, in the constraints we set $u \mathrel{\mathop:}= u'-u''$ and $v \mathrel{\mathop:}=v'-v''$, and
in the objective function we choose $u' \mathrel{\mathop:}= u^+$, $u'' \mathrel{\mathop:}= u^-$, $v''\mathrel{\mathop:}=
v^-$, $v' \mathrel{\mathop:}= v^+$. Since $a \leq b$ and $l \leq s$ imply that $a^T u'
- b^T u'' + l ^T v'' - s^T v' \leq a^T u^+ - b^T u^- + l ^T v^- - s^T
v^+$, this specification does not influence the optimum in case of
maximization.
\end{proof}

We close this section by a short consideration of maximization
problems:
\begin{equation}\label{vlp_max}
\tag{VLP$_\text{max}$} \textstyle{\max_C} Px \quad \text{s.t.}\quad
a \leq B x \leq b,\quad l \leq x \leq s.
\end{equation}
In this case we deal with a lower image $\P \mathrel{\mathop:}= P[S] - C$ of the primal
problem and an upper image $\D\mathrel{\mathop:}= D[T] + K$ of the dual problem, which
can be stated as
\begin{equation}\label{vlp_d_max}
	\text{$K$-maximize}\quad D(u,w,v) \quad \text{ s.t. } \quad
        B^T u = P^T w + v,\quad Y w \geq 0,\quad c^T w = 1
\end{equation}
with objective function
\begin{equation}\label{eq_ddd1}
	D(u,v,w)=\of{\frac{c_q}{|c_q|} w_1,\,\dots,\,\frac{c_q}{|c_q|}
          w_{q-1},\,\bar d(u,v)}^T,
\end{equation}
where
\begin{equation}
	\bar d(u,v) = b^T u^+ - a^T u^- + s^T v^- - l^T v^+.
\end{equation}
The ordering cone is again $K\mathrel{\mathop:}=\{y \in \R^q |\; y_1 = 0, \dots,
y_{q-1}= 0, y_q \geq 0\}$.

\begin{theorem} \label{th3}
	 Let $c \in \Int C$ such that $c_q \neq 0$. Then Theorem
         \ref{th1} holds for \eqref{vlp_max} and for the dual problem
         \eqref{vlp_d_max} if the following coupling function is used:
 \begin{equation} \label{phi1_max}
 	\varphi(y,y^*) : = -c_q \sum_{i=1}^{q-1} y_i y^*_i - y_q
        \of{\frac{|c_q|}{c_q} -\sum_{i=1}^{q-1} c_i y^*_i} + \xi(y)
        |c_q| y^*_q.
 \end{equation}
 \end{theorem}
\begin{proof}
	This follows from Theorem \ref{th2}. We first replace
        maximization with respect to $C$ by minimization with respect
        to $-C$. This requires the transformations $Y \to -Y$ and $c
        \to -c$. In the dual problem we substitute $(u,w,v) \to
        -(u,w,v)$. The resulting maximization problem with respect to
        $K$ is finally expressed as a minimization problem with
        respect to $K$, which refers to a coordinate transformation
        $y^*_q \to -y^*_q$.
\end{proof}

\section{A few remarks on the algorithms}
{\sl Bensolve} is an implementation of primal and dual Benson-type
algorithms. The origin of these algorithms is discussed in Section
1. The implementation is closely related to the presentation in
\cite{HamLoeRud13}, therefore, we do not present too much details
here. {\sl Bensolve} can handle problems which are generalized in
comparison to \cite{HamLoeRud13} with respect to the following two
aspects:
\begin{enumerate}
	\item The assumption $c_q = 1$ can be replaced by $c_q \neq
          0$, which allows to use arbitrary solid and line-free
          polyhedral ordering cones.
	\item Maximization is covered in addition to minimization.	
\end{enumerate}
These generalizations can be easily realized by the transformations
and the dual programs discussed in the previous section. A
modification of the algorithms is not necessary.

\section{Numerical results}
The VLP-solver \bensolve\ is a {\tt C}-implementation of the algorithms
discussed above.
In this section, we
investigate and compare numerical properties of two applications of
multiobjective optimization problems.\par
The first example (Example
\ref{ex.rttp}) is used to show the improvements that could be made in
comparison to prior implementations. The second problem class (Example
\ref{ex.csirmaz}) with image space dimension equal to ten shows that
\bensolve\ is well suited for problems with high image space
dimensions.\par
The numerical examples were run on a computer with 8GB memory and an
Intel\textregistered\ Core\texttrademark\ i5-4200 CPU with 1.60GHz
clock.  The source code of \bensolve\footnote{The source code is
available at {\tt http://bensolve.org} along with documentation and
example problems.}  was compiled with the  GNU Compiler Collection
{\tt gcc 5.2.1} and linked  against the GNU Linear Programming Kit
library {\tt libglp 4.55}.
\begin{ex}
  \label{ex.rttp}
  In their paper \cite{ShaEhr08}, Shao and Ehrgott use a variant of
  Benson's algorithm to solve MOLPs originating from a model for
  optimizing radiotherapy treatment planning.  The clinical example
  (PL) from this paper
  has three objectives,
  $1211$ constraints and $595$ variables (which are additionally
  nonnegative).  The constraint matrix is sparse with $153936$
  nonzeros.  Solving this problem exactly is not tractable.  Thus, an
  approximation variant of Benson's algorithm has been introduced
  in \cite{ShaEhr08}: The algorithm stops
  as soon as the translation of an approximation $\P^{\varepsilon}$
  of the upper image $\P$ by the duality parameter $c$ is contained in
  the upper image:
  \begin{equation*}
    \P^{\varepsilon}+\varepsilon\cdot c \subseteq \P\text{.}
  \end{equation*}\par
  This problem instance is also used for numerical tests in
  \cite{HamLoeRud13} with  a
  Matlab\textregistered-implementation preceding the current
  version of \bensolve.  The implementation used in \cite{HamLoeRud13}
  is similar to \oldsolve\footnote{\oldsolve\ is available in the
    download section at {\tt http://bensolve.org}}.  The major
  difference is the use of a certain warm start heuristic in \cite{HamLoeRud13}.
  Approximations for various error-levels $\varepsilon$ were computed
  with both the primal and dual variant of \bensolve\footnote{The
  authors thank Dr.~Lizhen~Shao for supplying the problem data.}, were we used
  the primal simplex as LP-solver in both cases.
  The run times are listed and compared with those of
  the implementation from \cite{HamLoeRud13} in Table \ref{tab.rttp_run}.  It
  can be seen that \bensolve\ performs superior to the
  preceding implementation with regard to the run times.  This might be
  explained partially by the speed-up gained through the transition
  from a Matlab\textregistered\ to a {\tt C}-based implementation or by refinements in the algorithm itself. It should be noted
  that even better results are expected as soon as a warm start
  heuristic similar to the one used in \cite{HamLoeRud13}
  is implemented.  Currently, a kind of ``indirect'' warm start technique is used,
  which is imminent in the LP-solver\footnote{GNU Linear Programming
    Kit, {\tt http://www.gnu.org/software/glpk/glpk.html}} we use for our
  implementation: between subsequent calls to the solver, the basis
  factorization of the solution found last remains in memory, hence
  speeding up the computation time due to the similar structure of the
  LP's.\par
  A further explanation for the improved run times is the
  utilization of a variant of the double description method (see e.g.\ \cite{MR1448924}), adapted
  specifically for the vertex enumeration part of \bensolve: In each iteration we are given an H-representation and a V-representation of a polyhedron $P$. We intend to compute the intersection $P\cap H$ of $P$ with an affine halfspace $H$. We assume that only a few vertices of $P$ do not belong to $H$. The classical variant requires a classification of vertices $K_1$ of $P$ belonging to $H$ and vertices $K_2$ of $P$ not belonging to $H$. The adjacency relation for every pair $(v_1,v_2)$, where $v_1 \in K_1$ and $v_2 \in K_2$, needs to be checked in order to compute the new V-representation. The new variant avoids this classification. One starts with some $v_1 \in K_1$ and checks every neighboring vertex $v_2$ whether it belongs to $K_2$ or not. If not, the procedure is repeated with $v_2$. 
  
  While approximations with an error threshold below $5\cdot 10^{-3}$
  where hardly possible with the version used in \cite{HamLoeRud13}, our new
  implementation is capable of finding approximations
  with error level $1\cdot 10^{-5}$ in a reasonable amount of
  time. Additionally, an improvement with respect to the
  approximation quality can be observed: both the primal and dual
  variants of \bensolve\ are able to achieve approximations of the
  upper image with the same approximation error $\varepsilon$ with
  fewer vertices (compare Table \ref{tab.rttp_pnts}) and by solving
  less linear programs (Table \ref{tab.rttp_lps}). This behavior can be explained by a new implementation of the vertex enumeration: If two vertices are very close to each other, they are replaced by a single one, while maintaining the outer approximation property.
  \begin{table}
    \caption{Run time comparison}
    \label{tab.rttp_run}
    \begin{center}
\begin{tabular}{@{}r@{\extracolsep{3em}}*{3}{l!{\extracolsep{0pt}}}@{}}\toprule
  & \multicolumn{3}{c}{Run times in seconds}\tabularnewline\cmidrule{2-4}
error $\varepsilon$ & \cite{HamLoeRud13} & primal & dual\tabularnewline
\midrule
$0.3$ & $47$ & $8$ & $4$ \tabularnewline
$0.1$ & $91$ & $15$ & $6$ \tabularnewline
$5\cdot 10^{-2}$ & $144$ & $22$ & $9$ \tabularnewline
$5\cdot 10^{-3}$ & $1411$ & $87$ & $36$ \tabularnewline
$5\cdot 10^{-4}$ &  & $355$ & $153$ \tabularnewline
$5\cdot 10^{-5}$ &  & $1318$ & $692$ \tabularnewline
$1\cdot 10^{-5}$ &  & $3376$ & $2185$ \tabularnewline\bottomrule
\end{tabular}
\end{center}
  \end{table}
  \begin{table}
    \caption{Number of vertices computed.}
    \label{tab.rttp_pnts}
    \begin{center}
\begin{tabular}{@{}r@{\extracolsep{3em}}*{6}{l!{\extracolsep{0pt}}}@{}}\toprule & \multicolumn{3}{c}{\# vertices of $\P$}& \multicolumn{3}{c}{\# vertices of $\D$}\tabularnewline\cmidrule(r){2-4}\cmidrule(l){5-7}
error $\varepsilon$ & \cite{HamLoeRud13} & primal & dual & \cite{HamLoeRud13} & primal & dual\tabularnewline
\midrule
$0.3$ & $46$ & $41$ & $21$ & $29$ & $29$ & $40$ \tabularnewline
$0.1$ & $104$ & $91$ & $54$ & $61$ & $60$ & $88$ \tabularnewline
$5\cdot 10^{-2}$ & $176$ & $156$ & $75$ & $94$ & $94$ & $146$ \tabularnewline
$5\cdot 10^{-3}$ & $1456$ & $1097$ & $548$ & $597$ & $595$ & $1078$ \tabularnewline
$5\cdot 10^{-4}$ &  & $7887$ & $3944$ &  & $4072$ & $7829$ \tabularnewline
$5\cdot 10^{-5}$ &  & $47293$ & $23877$ &  & $24190$ & $47310$ \tabularnewline
$1\cdot 10^{-5}$ &  & $117981$ & $65235$ &  & $65354$ & $118668$ \tabularnewline\bottomrule
\end{tabular}
\end{center}
  \end{table}
  \begin{table}
    \caption{Number of LP's solved}
    \label{tab.rttp_lps}
    \begin{center}
\begin{tabular}{@{}r@{\extracolsep{3em}}*{3}{l!{\extracolsep{0pt}}}@{}}\toprule & \multicolumn{3}{c}{\#LP's solved}\tabularnewline\cmidrule{2-4}
error $\varepsilon$ & \cite{HamLoeRud13} & primal & dual\tabularnewline
\midrule
$0.3$ & $75$ & $72$ & $66$ \tabularnewline
$0.1$ & $165$ & $157$ & $145$ \tabularnewline
$5\cdot 10^{-2}$ & $270$ & $258$ & $227$ \tabularnewline
$5\cdot 10^{-3}$ & $2053$ & $1772$ & $1710$ \tabularnewline
$5\cdot 10^{-4}$ &  & $12549$ & $12348$ \tabularnewline
$5\cdot 10^{-5}$ &  & $75656$ & $75419$ \tabularnewline
$1\cdot 10^{-5}$ &  & $194893$ & $195458$ \tabularnewline\bottomrule
\end{tabular}
\end{center}
  \end{table}
  \begin{figure}[ht]
    \centering
    \includegraphics[keepaspectratio=true,width=\textwidth]{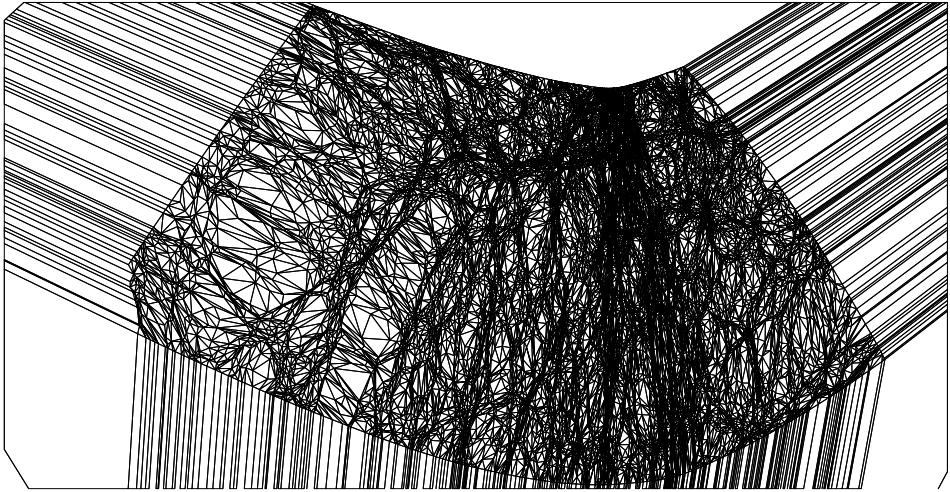}
    \caption{Approximation of the upper image of the problem instance
      (PL) from \cite{ShaEhr08} with $\varepsilon = 10^{-4}$}
  \end{figure}
\end{ex}
\begin{ex}
  \label{ex.csirmaz}
  In \cite{Csirmaz15} Csirmaz presents a rather fascinating
  application of Benson's algorithm: Exploring the entropy region
  formed by the entropies of the nonempty subsets of four random
  variables, which consists basically in finding a V-representation of
  a $10$-dimensional projection of a high dimensional polytope.
  Csirmaz used a revised version of Benson's algorithm
  to solve several instances of vector linear programs, which are
  generated by a procedure involving so called ``copy steps''.  The
  problem instances have been solved with \bensolve\footnote{The authors
  are indebted to Prof.~L\'aszl\'o Csirmaz for providing the {\tt perl}-script
  that generates the problem instances.} and we compared the
  run times with those of \cite{Csirmaz15} (compare table
  \ref{tab.csirmaz_run}).  Here we used the default options of
  \bensolve, the primal algorithm was used with an
  approximation error of $\varepsilon=10^{-8}$.  Although we used a
  machine with slightly higher specifications than the one used in \cite{Csirmaz15}, our
  implementation seems to be much faster, especially for large problem
  instances.  One explanation might be the adapted vertex enumeration
  method used in \bensolve: In \cite{Csirmaz15} it is stated that
  due to the huge number of vertices of intermediate polytopes the
  vertex enumeration (the double description method) becomes the
  bottleneck of the algorithm.
  It should be noted that the author of \cite{Csirmaz15} put a lot of
  effort in 
  reducing numerical errors to a minimum.  Therefore, we also list
  the number of vertices and facets computed by the different implementations
  in Table \ref{tab.csirmaz_pnts}.  While the number of primal
  vertices concur for all problem instances, the number of dual
  vertices (which corresponds to the number of facets of the upper image)
  differs slightly for different instances.  This may be the result of
  numerical errors or imprecision due to the approximating character
  of Benson's algorithm.
  \begin{table}
    \caption{Run time comparison of different instances of
      entropy region mapping problems}
    \label{tab.csirmaz_run}
    \begin{center}
\begin{tabular}{@{}l@{\extracolsep{3em}}*{2}{r!{\extracolsep{0pt}}}@{}}\toprule & \multicolumn{2}{c}{Run times in $hh\mathord{:}mm\mathord{:}ss$}\tabularnewline\cmidrule{2-3}
problem instance (copy string) & result of \cite{Csirmaz15} & \bensolve\tabularnewline
\midrule
$r\mathord{=}c\mathord{:}ab;s\mathord{=}r\mathord{:}ac;t\mathord{=}r\mathord{:}ad$ & $00\mathord{:}00\mathord{:}01$ & $00\mathord{:}00\mathord{:}01$ \tabularnewline
$rs\mathord{=}cd\mathord{:}ab;t\mathord{=}r\mathord{:}ad;u\mathord{=}s\mathord{:}adt$ & $00\mathord{:}06\mathord{:}19$ & $00\mathord{:}00\mathord{:}07$ \tabularnewline
$rs\mathord{=}cd\mathord{:}ab;t\mathord{=}a\mathord{:}bcs;u\mathord{=}(cs)\mathord{:}abrt$ & $00\mathord{:}06\mathord{:}51$ & $00\mathord{:}00\mathord{:}08$ \tabularnewline
$rs\mathord{=}cd\mathord{:}ab;t\mathord{=}a\mathord{:}bcs;u\mathord{=}b\mathord{:}adst$ & $00\mathord{:}17\mathord{:}40$ & $00\mathord{:}00\mathord{:}14$ \tabularnewline
$rs\mathord{=}cd\mathord{:}ab;t\mathord{=}a\mathord{:}bcs;u\mathord{=}t\mathord{:}acr$ & $00\mathord{:}18\mathord{:}27$ & $00\mathord{:}00\mathord{:}09$ \tabularnewline
$rs\mathord{=}cd\mathord{:}ab;t\mathord{=}(cr)\mathord{:}ab;u\mathord{=}t\mathord{:}acs$ & $00\mathord{:}22\mathord{:}58$ & $00\mathord{:}00\mathord{:}12$ \tabularnewline
$r\mathord{=}c\mathord{:}ab;st\mathord{=}cd\mathord{:}abr;u\mathord{=}a\mathord{:}bcrt$ & $00\mathord{:}29\mathord{:}18$ & $00\mathord{:}00\mathord{:}21$ \tabularnewline
$rs\mathord{=}cd\mathord{:}ab;t\mathord{=}a\mathord{:}bcs;u\mathord{=}c\mathord{:}abrt$ & $01\mathord{:}04\mathord{:}32$ & $00\mathord{:}00\mathord{:}49$ \tabularnewline
$rs\mathord{=}cd\mathord{:}ab;t\mathord{=}a\mathord{:}bcs;u\mathord{=}s\mathord{:}abcdt$ & $01\mathord{:}07\mathord{:}01$ & $00\mathord{:}00\mathord{:}43$ \tabularnewline
$rs\mathord{=}cd\mathord{:}ab;t\mathord{=}a\mathord{:}bcs;u\mathord{=}(at)\mathord{:}bcs$ & $01\mathord{:}39\mathord{:}30$ & $00\mathord{:}01\mathord{:}04$ \tabularnewline
$rs\mathord{=}cd\mathord{:}ab;t\mathord{=}a\mathord{:}bcs;u\mathord{=}a\mathord{:}bcst$ & $04\mathord{:}30\mathord{:}26$ & $00\mathord{:}04\mathord{:}51$ \tabularnewline
$rs\mathord{=}cd\mathord{:}ab;t\mathord{=}a\mathord{:}bcs;ua\mathord{:}bdrt$ & $05\mathord{:}11\mathord{:}25$ & $00\mathord{:}47\mathord{:}44$ \tabularnewline
$rs\mathord{=}cd\mathord{:}ab;tu\mathord{=}cr\mathord{:}ab;v\mathord{=}(cs)\mathord{:}abtu$ & $01\mathord{:}10\mathord{:}10$ & $00\mathord{:}01\mathord{:}27$ \tabularnewline
$rs\mathord{=}ad\mathord{:}bc;tu\mathord{=}ar\mathord{:}bc;v\mathord{=}r\mathord{:}abst$ & $03\mathord{:}24\mathord{:}37$ & $00\mathord{:}01\mathord{:}25$ \tabularnewline
$rs\mathord{=}cd\mathord{:}ab;t\mathord{=}(cr)\mathord{:}ab;uv\mathord{=}cs\mathord{:}abt$ & $03\mathord{:}34\mathord{:}31$ & $00\mathord{:}01\mathord{:}59$ \tabularnewline
$rs\mathord{=}cd\mathord{:}ab;tu\mathord{=}cr\mathord{:}ab;v\mathord{=}t\mathord{:}adr$ & $09\mathord{:}20\mathord{:}19$ & $00\mathord{:}02\mathord{:}21$ \tabularnewline
$rs\mathord{=}cd\mathord{:}ab;tu\mathord{=}dr\mathord{:}ab;v\mathord{=}b\mathord{:}adsu$ & $13\mathord{:}20\mathord{:}08$ & $00\mathord{:}03\mathord{:}01$ \tabularnewline
$rs\mathord{=}cd\mathord{:}ab;tv\mathord{=}dr\mathord{:}ab;u\mathord{=}a\mathord{:}bcrt$ & $14\mathord{:}34\mathord{:}42$ & $00\mathord{:}02\mathord{:}44$ \tabularnewline
$rs\mathord{=}cd\mathord{:}ab;tu\mathord{=}cs\mathord{:}ab;v\mathord{=}a\mathord{:}bcrt$ & $22\mathord{:}02\mathord{:}39$ & $00\mathord{:}02\mathord{:}42$ \tabularnewline
$rs\mathord{=}cd\mathord{:}ab;t\mathord{=}a\mathord{:}bcs;uv\mathord{=}bt\mathord{:}acr$ & $37\mathord{:}15\mathord{:}33$ & $00\mathord{:}08\mathord{:}27$ \tabularnewline
$rs\mathord{=}cd\mathord{:}ab;tu\mathord{=}cr\mathord{:}ab;v\mathord{=}a\mathord{:}bcstu$ & $427\mathord{:}43\mathord{:}30$ & $18\mathord{:}53\mathord{:}29$ \tabularnewline\bottomrule
\end{tabular}
\end{center}
  \end{table}
  \begin{table}
    \caption{Computed vertices and facets of entropy region mapping problems}
    \label{tab.csirmaz_pnts}
    \begin{center}
\begin{tabular}{@{}l@{\extracolsep{3em}}*{2}{l!{\extracolsep{0pt}}}@{}}\toprule & \multicolumn{2}{c}{\#vertices/\#facets}\tabularnewline\cmidrule{2-3}
problem instance (copy string) & result of \cite{Csirmaz15} & \bensolve\tabularnewline
\midrule
$r\mathord{=}c\mathord{:}ab;s\mathord{=}r\mathord{:}ac;t\mathord{=}r\mathord{:}ad$ & $5/20$ & $5/20$ \tabularnewline
$rs\mathord{=}cd\mathord{:}ab;t\mathord{=}r\mathord{:}ad;u\mathord{=}s\mathord{:}adt$ & $40/132$ & $40/133$ \tabularnewline
$rs\mathord{=}cd\mathord{:}ab;t\mathord{=}a\mathord{:}bcs;u\mathord{=}(cs)\mathord{:}abrt$ & $47/76$ & $47/77$ \tabularnewline
$rs\mathord{=}cd\mathord{:}ab;t\mathord{=}a\mathord{:}bcs;u\mathord{=}b\mathord{:}adst$ & $177/261$ & $177/263$ \tabularnewline
$rs\mathord{=}cd\mathord{:}ab;t\mathord{=}a\mathord{:}bcs;u\mathord{=}t\mathord{:}acr$ & $85/134$ & $85/136$ \tabularnewline
$rs\mathord{=}cd\mathord{:}ab;t\mathord{=}(cr)\mathord{:}ab;u\mathord{=}t\mathord{:}acs$ & $181/245$ & $181/247$ \tabularnewline
$r\mathord{=}c\mathord{:}ab;st\mathord{=}cd\mathord{:}abr;u\mathord{=}a\mathord{:}bcrt$ & $209/436$ & $209/438$ \tabularnewline
$rs\mathord{=}cd\mathord{:}ab;t\mathord{=}a\mathord{:}bcs;u\mathord{=}c\mathord{:}abrt$ & $363/599$ & $363/601$ \tabularnewline
$rs\mathord{=}cd\mathord{:}ab;t\mathord{=}a\mathord{:}bcs;u\mathord{=}s\mathord{:}abcdt$ & $355/591$ & $355/593$ \tabularnewline
$rs\mathord{=}cd\mathord{:}ab;t\mathord{=}a\mathord{:}bcs;u\mathord{=}(at)\mathord{:}bcs$ & $484/676$ & $484/677$ \tabularnewline
$rs\mathord{=}cd\mathord{:}ab;t\mathord{=}a\mathord{:}bcs;u\mathord{=}a\mathord{:}bcst$ & $880/1238$ & $880/1239$ \tabularnewline
$rs\mathord{=}cd\mathord{:}ab;t\mathord{=}a\mathord{:}bcs;ua\mathord{:}bdrt$ & $2506/2708$ & $2506/2710$ \tabularnewline
$rs\mathord{=}cd\mathord{:}ab;tu\mathord{=}cr\mathord{:}ab;v\mathord{=}(cs)\mathord{:}abtu$ & $19/58$ & $19/58$ \tabularnewline
$rs\mathord{=}ad\mathord{:}bc;tu\mathord{=}ar\mathord{:}bc;v\mathord{=}r\mathord{:}abst$ & $40/103$ & $40/101$ \tabularnewline
$rs\mathord{=}cd\mathord{:}ab;t\mathord{=}(cr)\mathord{:}ab;uv\mathord{=}cs\mathord{:}abt$ & $30/102$ & $30/102$ \tabularnewline
$rs\mathord{=}cd\mathord{:}ab;tu\mathord{=}cr\mathord{:}ab;v\mathord{=}t\mathord{:}adr$ & $167/235$ & $167/235$ \tabularnewline
$rs\mathord{=}cd\mathord{:}ab;tu\mathord{=}dr\mathord{:}ab;v\mathord{=}b\mathord{:}adsu$ & $318/356$ & $318/356$ \tabularnewline
$rs\mathord{=}cd\mathord{:}ab;tv\mathord{=}dr\mathord{:}ab;u\mathord{=}a\mathord{:}bcrt$ & $318/356$ & $318/356$ \tabularnewline
$rs\mathord{=}cd\mathord{:}ab;tu\mathord{=}cs\mathord{:}ab;v\mathord{=}a\mathord{:}bcrt$ & $297/648$ & $297/650$ \tabularnewline
$rs\mathord{=}cd\mathord{:}ab;t\mathord{=}a\mathord{:}bcs;uv\mathord{=}bt\mathord{:}acr$ & $779/1269$ & $779/1271$ \tabularnewline
$rs\mathord{=}cd\mathord{:}ab;tu\mathord{=}cr\mathord{:}ab;v\mathord{=}a\mathord{:}bcstu$ & $4510/7966$ & $4510/7972$ \tabularnewline\bottomrule
\end{tabular}
\end{center}
  \end{table}
\end{ex}

\clearpage
\section*{Bibliography}
\bibliographystyle{elsarticle-num}
\bibliography{database}

\end{document}